\documentclass[letterpaper, 10 pt, conference]{ieeeconf}  

\IEEEoverridecommandlockouts                              

\usepackage{amsthm}

\usepackage{microtype}
\usepackage{graphicx}
\usepackage{subfigure}
\usepackage{booktabs} 
\usepackage{algorithm}
\usepackage{algorithmic}
\usepackage[hidelinks]{hyperref}
\usepackage{url}
\usepackage{dsfont} 
\usepackage{siunitx}
\usepackage{amsmath,amssymb,amsfonts}
\usepackage{cleveref}
\usepackage{mathrsfs}
\usepackage{hyperref}
\usepackage{cite}

\newtheorem{proposition}{Proposition}

\newtheorem{remark}{Remark}
\newtheorem{theorem}{Theorem}

\title{\LARGE \bf
Convergence Analysis of the Best Response Algorithm \\ for Time-Varying Games
}

\author{Zifan Wang, Yi Shen, Michael M. Zavlanos, and Karl H. Johansson
\thanks{* This work was supported in part by Swedish Research Council Distinguished Professor Grant 2017-01078, Knut and Alice Wallenberg Foundation, Wallenberg Scholar Grant, the Swedish Strategic Research Foundation CLAS Grant RIT17-0046, AFOSR under award \#FA9550-19-1-0169, and  NSF under award CNS-1932011.
}
\thanks{Zifan Wang and Karl H. Johansson are with Division of Decision and Control Systems, School of Electrical Enginnering and Computer Science, KTH Royal Institute of Technology, and also with Digital Futures, SE-10044 Stockholm, Sweden. Email: \{zifanw,kallej\}@kth.se.}
\thanks{Yi Shen and Michael M. Zavlanos are with the Department of Mechanical Engineering and Materials Science, Duke University, Durham, NC, USA. Email: \{yi.shen478, michael.zavlanos\}@duke.edu}%
}

\begin{document}

\maketitle
\thispagestyle{empty}
\pagestyle{empty}

\begin{abstract}
This paper studies a class of strongly monotone games involving non-cooperative agents that optimize their own time-varying cost functions.
We assume that the agents can observe other agents' historical actions and choose actions that best respond to other agents' previous actions; we call this a best response scheme.
We start by analyzing the convergence rate of this best response scheme for standard time-invariant games.
Specifically, we provide a sufficient condition on the strong monotonicity parameter of the time-invariant games under which the proposed best response algorithm achieves exponential convergence to the static Nash equilibrium. 
We further illustrate that this best response algorithm may oscillate when the proposed sufficient condition fails to hold, which indicates that this condition is tight.
Next, we analyze this best response algorithm for time-varying games where the cost functions of each agent change over time.
Under similar conditions as for time-invariant games, we show that the proposed best response algorithm stays asymptotically close to the evolving equilibrium. We do so by analyzing both the equilibrium tracking  error and the dynamic regret.
Numerical experiments on economic market problems are presented to validate our analysis.

\end{abstract}

\section{Introduction}
Online convex games study the interplay between game theory and online learning, and find many applications ranging from traffic routing \cite{sessa2019no} to economic market optimization \cite{wang2022risk,lin2020finite}. 
In these games, agents simultaneously take actions to minimize their loss functions, which depend on the other agents' actions.

Generally, every agent in an online convex game adapts its actions to the actions of other agents in a dynamic manner with the objective to minimize its regret, defined as the cumulative difference in performance between the agent's online actions and the best single action in hindsight.
An algorithm is said to achieve no-regret learning if every agent's regret generated by this algorithm is sub-linear in the total number of episodes.
If the agents in an online game reach a stationary point from which no agent has an incentive to deviate, then we say the game has reached a Nash equilibrium.
There is a growing literature \cite{tatarenko2018learning,bravo2018bandit,drusvyatskiy2021improved,mertikopoulos2019learning,wang2022zeroth} that analyzes the Nash equilibrium convergence in strongly monotone games which admit a unique Nash equilibrium as shown in~\cite{rosen1965existence}.

In non-cooperative games, a common strategy used by competitive agents that selfishly minimize their own cost functions is the best response algorithm since it produces the most favorable outcome given the other agents plays.
The best response algorithm has been shown to converge under a spectral condition associated with the best-response map \cite{shanbhag2016inexact,facchinei201012}.
In general, best response algorithms have been studied for several classes of games, including supermodular games \cite{milgrom1990rationalizability}, potential games \cite{pass2019course,swenson2018best,lei2017randomized} and zero-sum games \cite{leslie2020best}. 
For example, \cite{swenson2018best} shows that in almost every potential game with finite actions, the best response dynamics converges to the unique Nash equilibrium with linear rate. 
Similarly, \cite{leslie2020best} shows the convergence of several best response dynamics in two-player zero-sum games.

In this paper, we study the regret and equilibrium tracking error of the best response algorithm for time-varying games. Specifically, we consider a class of strongly monotone games \cite{rosen1965existence,bravo2018bandit}, which guarantee the uniqueness of the well-defined Nash equilibrium. To the best of our knowledge, the best response algorithm has not been explored in the literature for time-varying games. 
Instead, time-varying games have been analyzed using  gradient-based algorithms for, e.g., strongly monotone games \cite{duvocelle2022multiagent} and zero-sum games \cite{zhang2022no}. Specifically, \cite{duvocelle2022multiagent} analyzes the Nash equilibrium convergence and the equilibrium tracking properties of the mirror descent algorithm for games that converge and diverge, respectively. In \cite{zhang2022no}, a gradient-type algorithm is proposed that achieves performance guarantees under three different measures.
As gradient-based algorithms are fundamentally different compared to the best response method, the techniques developed in these works cannot be applied here to analyze the best response algorithm.

To address this challenge, we first start with time-invariant games. Specifically, we assume games that satisfy the so-called strong monotonicity condition with parameter $m>0$, which guarantees the uniqueness of the Nash equilibrium \cite{rosen1965existence}. We provide a sufficient condition $m>L\sqrt{N-1}$ under which the best response algorithm achieves linear convergence to the static Nash equilibrium, where $L$ is the Lipschitz constant related to the gradient of the individual loss functions and $N$ is the number of agents. 
Moreover, we show numerically that when this condition fails to hold, the best response algorithm may oscillate. 
Compared to \cite{facchinei201012}, here we characterize the convergence in terms of the strong monotonicity parameter. For simple problems, we can show that our proposed condition is equivalent to the spectral condition proposed in \cite{facchinei201012}. 
Then, we analyze the best response algorithm for time-varying games where the Nash equilibrium evolves over time. Specifically, under similar conditions as for time-invariant games, we show that the average distance from the evolving equilibrium is bounded by the equilibrium variation. 
We also show that the dynamic regret is bounded by the cumulative variations of the loss functions.

The rest of the paper is organized as follows. In Section~\ref{sec:problem}, we provide some preliminaries and formally define the problem. In Section~\ref{sec:BR}, we present the regret and equilibrium convergence of the best response algorithm for time-invariant games. In Section~\ref{sec:TV_BR}, we extend our result to time-varying games and analyze the equilibrium tracking error and the dynamic regret. In Section~\ref{sec:simulation}, numerical experiments on  a Cournot game are presented to verify our method. Finally, in Section~\ref{sec:conclusion}, we conclude the paper.

\section{Preliminaries and Problem Definition}\label{sec:problem}
\subsection{Online Convex Games}
Consider an online convex game $\mathcal{G}$ with $N$ agents, whose goal is to learn their best individual actions that minimize their local loss functions.
For each agent  $i\in \mathcal{N}=\{1,\ldots,N\}$, denote by $\mathcal{C}_i(x_i,x_{-i}) : \mathcal{X} \rightarrow \mathbb{R}$ its individual loss function, where $x_i \in \mathcal{X}_i$ is the action of agent $i$, $x_{-i}$ are the actions of all agents excluding agent $i$, and we define by $\mathcal{X} =\Pi_{i=1}^N\mathcal{X}_i$ the joint action space since each agent takes actions independently. 
For ease of notation, we collect all agents' actions in a vector $x:=(x_1,\ldots,x_N)$. 
We assume that $\mathcal{C}_i(x)$ is convex in $x_i$ for all $x_{-i} \in \mathcal{X}_{-i}$, where $\mathcal{X}_{-i}$ is the joint action space excluding agent $i$.
%
The goal of every agent~$i$ is to determine the action $x_i$ that minimizes its individual cost function, i.e., 
\begin{align}\label{eq:def:game}
    \mathop{{\rm{min}}}_{x_i \in \mathcal{X}_i} \mathcal{C}_{i}(x_i,x_{-i}).
\end{align}
As shown in \cite{rosen1965existence}, convex games always have at least one Nash equilibrium. In what follows, we denote by $x^{*}$ a Nash equilibrium of the game \eqref{eq:def:game}. Then, for each agent $i$, we have $\mathcal{C}_i(x^{*})\leq \mathcal{C}_i(x_i,x_{-i}^{*}),$ $\forall x_i \in \mathcal{X}_i$, $i\in\mathcal{N}$. At this Nash equilibrium point, agents are strategically stable in the sense that each agent lacks incentive to change its action.
Since the agents' loss functions are convex, the Nash equilibrium can also be characterized by the first-order optimality condition, i.e., $\langle \nabla_{x_i} \mathcal{C}_i(x^{*}), x_i - x_i^{*} \rangle \geq 0, \; \forall x_i \in \mathcal{X}_i, i\in\mathcal{N},$ where $\nabla_{x_i} \mathcal{C}_i(x)$ is the partial derivative of the loss function with respect to each agent's action.  We write $\nabla_{i} \mathcal{C}_i(x)$ instead of $\nabla_{x_i} \mathcal{C}_i(x)$ whenever it is clear from the context. 


In general, it is not easy to show convergence to a Nash equilibrium for games with multiple Nash Equilibria. 
For this reason, recent studies often focus on games that are so-called strongly monotone and are well-known to have a unique Nash equilibrium \cite{rosen1965existence}. 
The game \eqref{eq:def:game} is said to be $m$-strongly monotone if for $\forall x,x'\in \mathcal{X}$ we have that 
\begin{align}\label{eq:strong_monotone}
    \sum_{i=1}^N \langle \nabla_i \mathcal{C}_i(x) -\nabla_i \mathcal{C}_i(x'),x_i-x_i' \rangle \geq m \left\|x -x' \right\|^2.
\end{align}
The ability of the agents to efficiently learn their optimal actions can be quantified using the notion of (static) regret that captures the cumulative loss of the learned online actions compared to the best actions in hindsight, and can be formally defined as
\begin{align}\label{eq:def:regret:game}
    {\rm{SR}}_i(T)= \sum_{t=1}^T \mathcal{C}_i(x_t) - \mathop{\rm{min}}_{x_i} \sum_{t=1}^T\mathcal{C}_i(x_i,x_{-i,t}),
\end{align}
for sequences of actions $\{x_{i,t} \}_{t=1}^T, i=1,\ldots,N$.
An algorithm is said to be no-regret if the regret of each agent is sub-linear in the total number of episodes $T$, i.e., ${\rm{SR}}_i(T)=\mathcal{O}(T^a), a\in[0,1)$, $\forall i \in \mathcal{N}$.

\subsection{Problem Definition}
In this work, we consider the time-varying game $\mathcal{G}_t$ where at episode $t$ each agent aims to minimize its time-varying cost function, i.e.,
\begin{align}\label{eq:def:TV:game}
    \mathop{{\rm{min}}}_{x_i \in \mathcal{X}_i} \mathcal{C}_{i,t}(x_i,x_{-i}).
\end{align}
Then, we can define the best response algorithm for time-varying games as
\begin{align}\label{eq:TVBR:update}
    x_{i,t+1} = \mathop{\rm{arg min}}_{x_i \in \mathcal{X}_i} \mathcal{C}_{i,t} (x_i, x_{-i,t}).
\end{align}
To attain the best response action $x_{i,t+1}$, for each agent $i$, we assume the cost function $\mathcal{C}_{i,t}$ is known and all other agents' previous actions are provided. This is not a very strong assumption. For example, in supply chain problems \cite{cachon2006game}, $\mathcal{C}_{i,t}$ can represent an agent's local revenue model that depends on all competitors' actions and unknown market demands. At the beginning of episode $t+1$, the agents may not be able to observe the other agents' actions and precisely predict the market demands. However, previous actions and demands can be obtained from public revenue reports. 
Thus, it is reasonable to implement a strategy where the agents take actions that best respond to the other agents' actions from the previous episode. 
In addition, we assume that at every episode $t$, the time-varying game with the cost function $\mathcal{C}_{i,t}$ is strongly monotone and thus has a unique Nash equilibrium, which we denote by $x_t^{*}$.
To analyze the performance of the best response algorithm \eqref{eq:TVBR:update} for time-varying games, we define the equilibrium tracking error 
\begin{align}\label{eq:BRTV:trackingerror}
    {\rm{Err}}(T):=\sum_{t=1}^T\left\| x_t - x_t^{*}\right\|^2,
\end{align}
and the dynamic regret
\begin{align}\label{eq:BRTV:dynamic:regret}
    {\rm{DR}}_i(T) := \sum_{t=1}^T \Big( \mathcal{C}_{i,t}(x_t) - \mathop{\rm{min}}_{y_i} \mathcal{C}_{i,t}(y_i,x_{-i,t})\Big),
\end{align}
where $T$ is the total number of episodes.
If the game $\mathcal{G}_t$ changes significantly over time, it is reasonable to expect that it may become impossible to track the evolving equilibrium. 
The time-varying  problem becomes meaningful only when the variation of the game $\mathcal{G}_t$ is reasonably small.
To capture the effect of the variation of the game $\mathcal{G}_t$ on the performance of the best response algorithm, we first define the equilibrium variation
\begin{align}\label{eq:def:VT}
    V_T:=\sum_{t=1}^T\left\|x_{t}^{*}- x_{t+1}^{*}  \right\|^2,
\end{align}
which tracks the changes of Nash equilibria.
It is possible that the cost function $ \mathcal{C}_{i,t}$ changes over time but the equilibrium stays constant, i.e., $V_T=0$. 
To further capture the variations of the cost functions, we define the function variation 
\begin{align}\label{eq:def:WT}
    W_{i,T} = \sum_{t=1}^T \sup_{x\in\mathcal{X}}|C_{i,t}(x) - C_{i,t+1}(x)|.
\end{align}
%
Our goal in this paper is to analyze the equilibrium tracking error and the dynamic regret of the best response algorithm \eqref{eq:TVBR:update} for time-varying games. To do so, we start with the analysis of time-invariant games and then extend our results to the time-varying case.

\section{Time-Invariant Games}\label{sec:BR}
In this section, we provide sufficient conditions for Nash equilibrium convergence of the best response algorithm for time-invariant games. The best response algorithm in this case becomes 
\begin{align}\label{eq:BR:update}
    x_{i,t+1} = \mathop{\rm{arg min}}_{x_i \in \mathcal{X}_i} \mathcal{C}_{i} (x_i, x_{-i,t}).
\end{align} 
%
%
%
\begin{proposition}\label{prop:BR}
Suppose that the game $\mathcal{G}$ is $m$-strongly monotone, and $\nabla_i \mathcal{C}_i(x_i,x_{-i})$ is $L$-Lipschitz continuous in $x_{-i}$ 
for every $x_i \in \mathcal{X}_i$, with parameter $m>L \sqrt{N-1}$. 
Then, the best response algorithm \eqref{eq:BR:update} satisfies that
\begin{align}\label{eq:BR:convergence}
    \left\| x_T - x^{*}\right\| \leq \rho^{T-1} \left\| x_1 - x^{*}\right\|,
\end{align}
where $\rho:= \frac{L\sqrt{N-1}}{m}$.
\end{proposition}
\begin{proof}
Applying the first order optimality condition to the cost function $\mathcal{C}_i$  at the optimal point $x_{i,t+1}$ and using the update rule \eqref{eq:BR:update}, we have that
\begin{align}\label{eq:BR_temp1}
    \langle \nabla_i \mathcal{C}_i (x_{i,t+1},x_{-i,t}),x_i - x_{i,t+1} \rangle \geq 0, \; \; \forall x_i \in \mathcal{X}_i.
\end{align}
Since the game is strongly monotone, we have that for all $x_i \in \mathcal{X}_i$,
\begin{align}\label{eq:BR_temp2}
    \langle \nabla_i \mathcal{C}_i (x_{i},x_{-i,t})- \nabla_i \mathcal{C}_i (x_{i,t+1},x_{-i,t}), x_i - x_{i,t+1} \rangle  \nonumber \\
    \geq m \left\| x_i -x_{i,t+1} \right\|^2,
\end{align}
which follows from the definition \eqref{eq:strong_monotone} by setting $x=(x_{i},x_{-i,t})$ and $x'=(x_{i,t+1},x_{-i,t})$.
Combining \eqref{eq:BR_temp2} with \eqref{eq:BR_temp1} and replacing $x_i$ with $x_i^{*}$, we get
\begin{align}\label{eq:BR_temp3}
    &m \left\| x_i^{*} -x_{i,t+1} \right\|^2  
    \leq  \langle \nabla_i \mathcal{C}_i (x_{i}^{*},x_{-i,t}) , x_i^{*} - x_{i,t+1} \rangle .
\end{align}
Summing the both sides of inequality \eqref{eq:BR_temp3} over $i=1,\ldots,N$, we have that
\begin{align}\label{eq:BR_temp4}
    & \left\| x_{t+1} - x^{*} \right\|^2 
    \leq  \frac{1}{m} \sum_i \langle \nabla_i \mathcal{C}_i (x_{i}^{*},x_{-i,t}) , x_i^{*} - x_{i,t+1} \rangle \nonumber \\
    & \leq  \frac{1}{m} \sum_i \langle \nabla_i \mathcal{C}_i (x_{i}^{*},x_{-i,t}) - \nabla_i \mathcal{C}_i (x^{*}) , x_i^{*} - x_{i,t+1} \rangle \nonumber \\
    & \leq   \frac{1}{m} \sum_i L \left\| x_{-i,t} - x_{-i}^{*}\right\| \left\| x_{i}^{*} - x_{i,t+1}\right\| \nonumber \\
   & \leq   \frac{L\sqrt{N-1}}{m}   \left\| x_{t} - x^{*}\right\| \left\| x^{*} - x_{t+1}\right\|,
\end{align}
where the second inequality follows from the Nash equilibrium condition $\langle \nabla_i \mathcal{C}_i (x^{*}),x_i - x_{i}^{*} \rangle \geq 0 $, $\forall x_i \in \mathcal{X}_i$ and the third inequality is due to the Lipschitz continuous property of the function $\mathcal{C}_i$ in $x_{-i}$.
The last inequality follows from the Cauchy-Schwarz inequality.
Dividing the inequality \eqref{eq:BR_temp4} by $\left\| x_{t+1} - x^{*} \right\|$ yields 
\begin{align}\label{eq:BR_temp5}
    \left\| x_{t+1} - x^{*} \right\| \leq \frac{L \sqrt{N-1}}{m}\left\| x_{t} - x^{*} \right\|.
\end{align}
Note, if $\left\| x_{t+1} - x^{*} \right\| = 0$, then \eqref{eq:BR_temp5} holds trivially. Applying inequality \eqref{eq:BR_temp5} iteratively over $t=1,\ldots,T-1$ completes the proof.
\end{proof}
In what follows, we provide some intuition and explain the condition $m>L \sqrt{N-1}$.
First, suppose that $L_1$ is the Lipschitz constant of the function $ \nabla_i \mathcal{C}_i(x)$ with respect to $x$. From its definitions we conclude that $L\leq L_1$. Therefore, the Lipschitz constant $L_1$ provides an upper bound on the variation of the gradients and is always greater than the strongly monotone parameter $m$ which provides a lower bound, i.e., $m\leq L_1$. However, it is still possible to have $m>L \sqrt{N-1} $. For example, if $\mathcal{C}_i$ only depends on $x_i$, we have that $L = 0$ and thus the condition naturally holds as long as $m>0$.

On the other hand, consider the condition $m>L\sqrt{N-1}$ and rearrange the terms to get $L<\frac{m}{ \sqrt{N-1}}$. Recall that $L$ is the Lipschitz constant of the function $\nabla_i \mathcal{C}_i(x_i,x_{-i})$ with respect to $x_{-i}$, which can be interpreted as the maximum influence of the other agents' actions on agent $i$. The condition $L<\frac{m}{ \sqrt{N-1}}$ requires that this influence is small enough for the game to converge.  The presence of multiple agents ($N$ is large) reduces the upper bound on the influence of other agents' actions which , effectively, increases the difficulty of the game.

Note that \cite{facchinei201012} also provides a sufficient condition for convergence of the best response algorithm, that involves the spectral norm of a matrix composed of parameters related to the second-order partial derivative of the cost function.
In this work, we analyze the best response algorithm from a different perspective that relies on strong monotonicity to characterize  convergence.
In simple cases such as two-player potential games, it is easy to show that our condition is equivalent to the condition in \cite{facchinei201012}.
However, in general, strong monotonicity  provides a more intuitive condition for convergence. Finally, we experimentally show that when the condition $m>L\sqrt{N-1}$ does not hold, the best-response algorithm may lead to cycles. This result further validates the utility of the proposed condition. 

Proposition \ref{prop:BR} shows that the best response algorithm converges to the Nash equilibrium at an exponential rate. Indeed, it is a no-regret learning algorithm for each agent as well, as shown in the following proposition.

\begin{proposition}\label{prop:BR:no_regret}
Suppose that the game $\mathcal{G}$ is $m$-strongly monotone with parameter $m>L \sqrt{N-1}$, the cost $C_i(x_i,x_{-i})$ is $L_0$-Lipschitz continuous in $x_{-i}$ for every $x_i \in \mathcal{X}_i$, and the diameter of the convex set $\mathcal{X}_i$ is bounded by $D$, for all $i=1,\ldots,N$ Then, the static regret of the best response algorithm satisfies
\begin{align*}
    {\rm{SR}}_i(T) \leq  \sum_{t=1}^T \mathcal{C}_i(x_t) -  \sum_{t=1}^T \min_{x_i}\mathcal{C}_i(x_i,x_{-i,t}) = \mathcal{O}(1).
\end{align*}
\end{proposition}

\begin{proof}
The first inequality holds due to the fact that $\sum_{t=1}^T \min_{x_i}\mathcal{C}_i(x_i,x_{-i,t})\leq \min_{x_i}\sum_{t=1}^T \mathcal{C}_i(x_i,x_{-i,t})$.
Observe that $\mathcal{C}_i(x_{i,t+1},x_{-i,t})=\min_{x_i}\mathcal{C}_i(x_i,x_{-i,t})$ since $x_{i,t+1}=\mathop{\rm{arg min}}_{x_i \in \mathcal{X}_i} \mathcal{C}_{i} (x_i, x_{-i,t})$. Then, it follows that
\begin{align}\label{eq:BR_no_temp1}
    &{\rm{SR}}_i(T) \leq \sum_{t=1}^T \mathcal{C}_i(x_t) -  \sum_{t=1}^T \min_{x_i}\mathcal{C}_i(x_i,x_{-i,t}) \nonumber\\
    & = \sum_{t=1}^T \Big( \mathcal{C}_i(x_t)- \mathcal{C}_i(x_{t+1})+ \mathcal{C}_i(x_{t+1}) - \mathcal{C}_i(x_{i,t+1},x_{-i,t}) \Big) \nonumber \\
    & \leq \mathcal{C}_i(x_1) + \sum_{t=1}^T \Big(\mathcal{C}_i(x_{t+1}) - \mathcal{C}_i(x_{i,t+1},x_{-i,t}) \Big) \nonumber \\
    & \leq  \mathcal{C}_i(x_1) + L_0 \sum_{t=1}^T \left\|x_{-i,t+1} - x_{-i,t}\right\| \nonumber \\
    & \leq  \mathcal{C}_i(x_1) + L_0 \sum_{t=1}^T \left\|x_{t+1} - x_{t}\right\|,
\end{align}
where the second to the last inequality follows from the Lipschitz continuous property of the function $\mathcal{C}_i$ in $x$.
By virtue of \eqref{eq:BR:convergence} in Proposition 1, we have
\begin{align}\label{eq:BR_no_temp2}
    &\left\|x_{t+1} - x_{t}\right\|^2 = \left\|x_{t+1} -x^{*}+x^{*} -x_{t}\right\|^2 \nonumber \\
    &\leq  2 \left\|x_{t+1} -x^{*}\right\|^2 + 2 \left\|x^{*} -x_{t}\right\|^2 
    \leq  2(\rho^2+1)\left\|x_{t}-x^{*} \right\|^2.
\end{align}
Substituting the inequality \eqref{eq:BR_no_temp2} into \eqref{eq:BR_no_temp1}, we have 
\begin{align}\label{eq:BR_no_temp3}
    {\rm{SR}}_i(T) 
    \leq & \mathcal{C}_i(x_1) + L_0 \sum_{t=1}^T \sqrt{2(\rho^2+1)} \left\|x_{t} - x^{*}\right\|  \nonumber \\
    \leq &  \mathcal{C}_i(x_1) + L_0 \sqrt{2(\rho^2+1)} \sum_{t=1}^T \rho^t D \nonumber \\
    \leq & \mathcal{C}_i(x_1)+\frac{ DL_0 \sqrt{2(\rho^2+1)}} {1-\rho},
\end{align}
which completes the proof.
\end{proof}

Proposition \ref{prop:BR:no_regret} indeed provides a stronger bound than the static regret defined in \eqref{eq:def:regret:game}. Instead of comparing to a single best action in hindsight, it compares with a sequence of episode-wise best actions, which is equivalent to the dynamic regret with time-invariant cost functions. This strong result own itself to the best response algorithm.

\section{Time-varying games}\label{sec:TV_BR}
In this section, we analyze time-varying games $\mathcal{G}_t$ where the cost functions of the agents change over time. Since the equilibrium of these games also varies, in what follows we analyze the ability of the best response algorithm \eqref{eq:TVBR:update} to generate actions that track the evolving equilibrium. 

If the game $\mathcal{G}_t$ changes significantly, it is reasonable to expect that it will be hard to track the evolving equilibrium. Therefore, as in related literature \cite{duvocelle2022multiagent,zhang2022no}, we assume that both the equilibrium variation $V_T$ in \eqref{eq:def:VT} and the function variation $W_{i,T}$ in \eqref{eq:def:WT} are sub-linear in $T$, for $i=1,\ldots,N$.

In what follows, we analyze the equilibrium tracking error of the best response algorithm \eqref{eq:TVBR:update} in terms of the equilibrium variation.
\begin{theorem}\label{theorem:BRTV}
Suppose that the time-varying game $\mathcal{G}_t$ is $m_t$-strongly monotone and $\nabla_i \mathcal{C}_{i,t}(x_i,x_{-i})$ is $L_t$-Lipschitz continuous in $x_{-i}$ for every $x_i \in \mathcal{X}_i$ with parameter $m_t>L_t \sqrt{N-1}$, for $\forall t$. Then, the best response algorithm \eqref{eq:TVBR:update} satisfies that
\begin{align}\label{eq:BRTV:convergence}
    {\rm{Err}}(T) \leq \frac{\left\| x_{1}- x_{1}^{*}\right\|^2}{1-\rho_m} + \frac{V_T}{(1-\rho_m)^2}=\mathcal{O}\left( 1+ V_T\right),
\end{align}
where $\rho_m:= \mathop{\rm{max}}_t \left\{ \frac{L_t\sqrt{N-1}}{m_t} \right\}$.
\end{theorem}

\begin{proof}
Applying the same arguments as in Proposition \ref{prop:BR} to the cost function $\mathcal{C}_{i,t}$, we can obtain an inequality similar to \eqref{eq:BR_temp5} as
\begin{align}\label{eq:BRTV_temp4}
    \left\| x_{t+1}-x_{t}^{*}  \right\| 
    \leq  \rho_t  \left\| x_{t} - x_{t}^{*}\right\|,
\end{align}
where $\rho_t:=  \frac{L_t\sqrt{N-1}}{m_t} $.
Observe that 
\begin{align*}
    &\left\| x_{t+1}-x_{t+1}^{*}  \right\|^2 = \left\| x_{t+1}- x_{t}^{*} + x_{t}^{*}- x_{t+1}^{*}  \right\|^2 \nonumber \\
    \leq &(1+\lambda) \left\| x_{t+1}- x_{t}^{*}\right\|^2 + (1+\frac{1}{\lambda}) \left\|x_{t}^{*}- x_{t+1}^{*}  \right\|^2,
\end{align*}
for $\forall \lambda >0$. Setting $\lambda = \frac{1}{\rho_t}-1>0$ yields 
\begin{align}\label{eq:BRTV_temp5}
    &\left\| x_{t+1}-x_{t+1}^{*}  \right\|^2 
    \leq \frac{1}{\rho_t } \left\| x_{t+1}- x_{t}^{*}\right\|^2 + \frac{1}{1-\rho_t } \left\|x_{t}^{*}- x_{t+1}^{*}  \right\|^2 \nonumber \\
    & \leq  \rho_t \left\| x_{t}- x_{t}^{*}\right\|^2 + \frac{1}{1-\rho_t } \left\|x_{t}^{*}- x_{t+1}^{*}  \right\|^2 \nonumber \\
    & \leq  \rho_m \left\| x_{t}- x_{t}^{*}\right\|^2 + \frac{1}{1-\rho_m } \left\|x_{t}^{*}- x_{t+1}^{*}  \right\|^2,
\end{align}
where the second inequality follows from \eqref{eq:BRTV_temp4} and the last inequality is due to the fact that $\rho_t \leq \rho_m <1$. Rearranging and summing \eqref{eq:BRTV_temp5} over $t=1,\ldots,T$, we have that
\begin{align*}
    &(1-\rho_m) \sum_{t=1}^T \left\| x_{t}- x_{t}^{*}\right\|^2 \nonumber \\
    \leq & \sum_{t=1}^T \left(\left\| x_{t}- x_{t}^{*}\right\|^2 - \left\| x_{t+1}-x_{t+1}^{*}  \right\|^2 + \frac{\left\|x_{t}^{*}- x_{t+1}^{*}  \right\|^2}{1-\rho_m } \right) \nonumber \\
    \leq & \left\| x_{1}- x_{1}^{*}\right\|^2 +\frac{1}{1-\rho_m } \sum_{t=1}^T\left\|x_{t}^{*}- x_{t+1}^{*}  \right\|^2 \nonumber \\
    \leq & \left\| x_{1}- x_{1}^{*}\right\|^2 +  \frac{1}{1-\rho_m } V_T.
\end{align*}
Dividing both sides of the above inequality by $(1-\rho_m)$ completes the proof.
\end{proof}

Theorem~\ref{theorem:BRTV} shows that $V_T$ dominates the equilibrium tracking error. If $V_T$ is sub-linear in $T$, so is the equilibrium tracking error.
In what follows, we analyze the dynamic regret of each agent in terms of the equilibrium variation and the function variation.
\begin{theorem}\label{theorem:BRTV:no_regret}
Suppose that the time-varying game $\mathcal{G}_t$ is $m_t$-strongly monotone, $\nabla_i \mathcal{C}_{i,t}(x_i,x_{-i})$ is $L_t$-Lipschitz continuous in $x_{-i}$ for every $x_i \in \mathcal{X}_i$ with parameter $m_t>L_t \sqrt{N-1}$, and the cost $\mathcal{C}_{i,t}(x)$ is $L_0$-Lipschitz continuous in $x_{-i}$ for every $x_i \in \mathcal{X}_i$ for $\forall t$. Then, the dynamic regret of the best response algorithm \eqref{eq:TVBR:update} satisfies
\begin{align}
    {\rm{DR}}_i(T)  = \mathcal{O}\left( W_{i,t} +\sqrt{TV_T} \right), \; i=1,\ldots,N.
\end{align}
\end{theorem}

\begin{proof}
Using the update rule of the best response algorithm \eqref{eq:TVBR:update}, we have 
\begin{align}
    &{\rm{DR}}_i(T) = \sum_{t=1}^T \Big( \mathcal{C}_{i,t}(x_t) -\mathcal{C}_{i,t}(x_{i,t+1},x_{-i,t})\Big) \nonumber \\
    =& \sum_{t=1}^T \Big( \mathcal{C}_{i,t}(x_t) - \mathcal{C}_{i,t+1}(x_{t+1}) + \mathcal{C}_{i,t+1}(x_{t+1}) \nonumber \\ 
    &- \mathcal{C}_{i,t}(x_{t+1}) + \mathcal{C}_{i,t}(x_{t+1}) - \mathcal{C}_{i,t}(x_{i,t+1},x_{-i,t}) \Big) \nonumber \\ 
    \leq & \mathcal{C}_{i,1}(x_1) +  W_{i,T} + \sum_{t=1}^T \Big(\mathcal{C}_{i,t}(x_{t+1}) - \mathcal{C}_{i,t}(x_{i,t+1},x_{-i,t}) \Big) \nonumber \\ 
    \leq & \mathcal{C}_{i,1}(x_1)+ W_{i,T} + L_0 \sum_{t=1}^T\left\| x_{-i,t+1}-x_{-i,t}\right\| \nonumber \\ 
    \leq & \mathcal{C}_{i,1}(x_1)+ W_{i,T} + L_0 \sum_{t=1}^T \left\| x_{t+1}-x_{t}\right\|.
\end{align}
Using the inequality \eqref{eq:BRTV_temp4} and the fact that $\rho_t \leq \rho_m <1$, we have
\begin{align*}
    &\sum_{t=1}^T\left\| x_{t+1}-x_{t}\right\|^2 =\sum_{t=1}^T\left\| x_{t+1}- x_t^{*} +x_t^{*} - x_{t}\right\|^2 \nonumber \\
    & \leq \sum_{t=1}^T \big( (1+\frac{1}{\rho_m}) \left\| x_{t+1}- x_t^{*}\right\|^2 + (1+\rho_m)  \left\|  x_t^{*} - x_{t}\right\|^2 \big) \nonumber \\
    & \leq (\rho_m+1)^2 \sum_{t=1}^T \left\|  x_{t} - x_t^{*} \right\|^2,
\end{align*}
which further yields
\begin{align}
    &{\rm{DR}}_i(T)  \nonumber \\
    \leq & \mathcal{C}_{i,1}(x_1)+ W_{i,T} + L_0 \sqrt{T} \sqrt{\sum_{t=1}^T \left\| x_{t+1}-x_{t}\right\|^2} \nonumber \\
    \leq &  \mathcal{C}_{i,1}(x_1)+ W_{i,T} + L_0 \sqrt{T} \sqrt{(\rho_m+1)^2 \sum_{t=1}^T \left\|  x_{t} - x_t^{*} \right\|^2} \nonumber \\
    =& \mathcal{O}\left( W_{i,t} +\sqrt{TV_T} \right),
\end{align}
where in the last inequality we use the results from Theorem~\ref{theorem:BRTV}.
The proof is complete.
\end{proof}

Theorem~\ref{theorem:BRTV:no_regret} shows that the dynamic regret is sublinear in $T$ if the variation of the game satisfies $W_{i,T}=\mathcal{O}(T^a)$ and $V_{T}=\mathcal{O}(T^b)$ with $a,b\in[0,1)$. 
\begin{remark}
(Connection between dynamic regret and equilibrium tracking error). In the single agent case, equilibrium tracking error is equivalent to the dynamic regret. However, this is not true for games involving multiple agents. This is due to the fact that the function $\mathcal{C}_{i,t}(\cdot,x_{-i,t})$ is time-varying due to changes in the function $\mathcal{C}_{i,t}$ itself and changes in other agents' actions $x_{-i,t}$.
To see this, consider the class of time-varying games with time-varying cost functions but constant equilibrium, i.e., $V_T=0$, $W_{i,T}=\mathcal{O}(T^a)$ for some $a>0$. In this case, we have $ {\rm{Err}}(T) = \mathcal{O}(1)$ but  ${\rm{DR}}_i(T) = \mathcal{O}(T^a)$.
\end{remark}

\section{Numerical Experiments}\label{sec:simulation}
In this section, we validate our analysis on a Cournot game for both time-invariant and time-varying losses.
\subsection{Time-invariant game}
We first focus on the time-invariant case.
We consider a Cournot game with two agents whose goal is to minimize their local losses by appropriately setting the production quantity $x_i$, $i=1,2$.
The loss function of each agent is given by $\mathcal{C}_i(x) = x_i(\frac{a_i x_i}{2} + b_i x_{-i} - e_i)+ 1$, where $a_i>0$ , $b_i$, $e_i$ are constant parameters, and $x_{-i}$ denotes the production quantity of the opponent of agent~$i$. 
It is easy to show that $\nabla_i \mathcal{C}_i(x) = a_i x_i + b_i x_{-i} -e_i$. Recalling that $L$ is the Lipschitz constant of the function $\nabla_i \mathcal{C}_i(x)$ with respect to $x_{-i}$, we have $L={\rm{max}}\{ |b_1|,|b_2|\}$.
Define $g(x) = (\nabla_1 \mathcal{C}_1(x), \nabla_2 \mathcal{C}_2(x))$ and let $G(x)$ denote the Jacobian of $g(x)$, i.e., $G(x)=[a_1,b_1;b_2,a_2]$. 
According to \cite{rosen1965existence}, the strong monotonicity parameter $m$  coincides with the smallest eigenvalue of the matrix $\frac{G(x)+G'(x)}{2}$. 

We validate our methods for three different selections of parameters $\theta^k:=(a_1^k,a_2^k,b_1^k,b_2^k,e_1^k,e_2^k)$ for $k=1,2,3$. Specifically, We select $\theta^1 = (1,1,0.6,-0.5,1.2,0.8)$, $\theta^2 = (1,1,1,-1,1.2,0.8)$ and $\theta^3 = (1,1,2,-1,1.2,0.8)$. It is easy to verify that $\theta^1$, $\theta^2$ and $\theta^3$ correspond to the cases $m>L\sqrt{N-1}$, $m=L\sqrt{N-1}$, and $m<L\sqrt{N-1}$, respectively. The convergence results are shown in Figure~\ref{fig_TI}. We observe that when $m>L\sqrt{N-1}$, the best response converges with exponential rate. When $m\leq L\sqrt{N-1}$, the best response algorithm fails to converge, which indicates the tightness of our theoretical results.

\begin{figure}[t] 
\begin{center}
\centerline{\includegraphics[width=0.9\columnwidth]{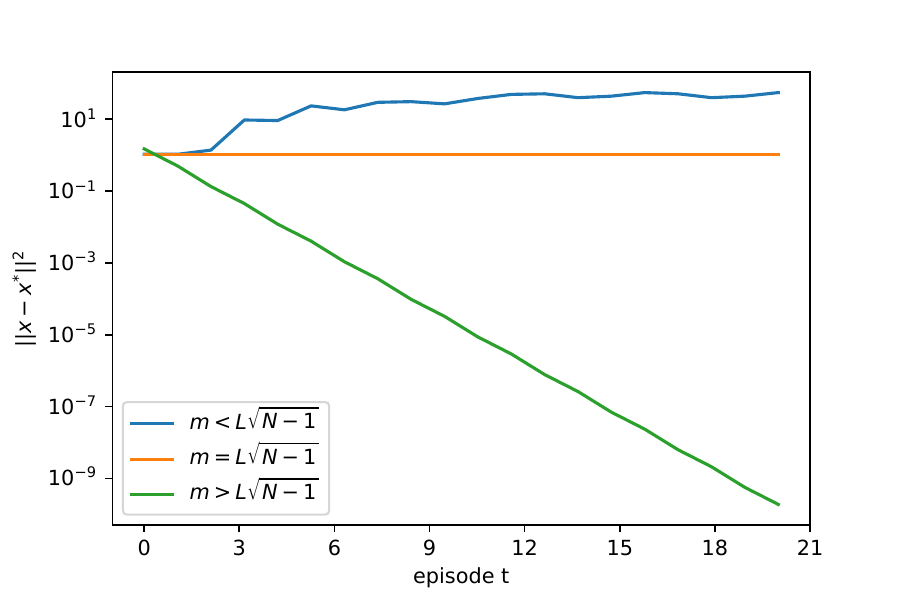}}
\caption{Convergence of the best response algorithm for time-invariant games.}
\label{fig_TI}
\end{center}
\vskip -0.2in
\end{figure}

\subsection{Time-varying games}
For the time-varying case, the loss function of agent $i$ is defined as $\mathcal{C}_{i,t}(x) = x_i(\frac{a_i x_i}{2} + b_{i,t} x_{-i} - e_{i,t})+ 1$, where $a_i=2$, $i=1,2$, and $b_{i,t}$, $e_{i,t}$ are time-varying parameters. The time-varying parameters are selected as 
\begin{align*}
    b_{i,t}=\left\{ \begin{array}{cc} 0.3+ 0.1\times (-1)^t  & t\in[1,T^{0.6}] \\
    0.3 & t\in(T^{0.6},T] \end{array},\right.
\end{align*}
\begin{align*}
    e_{i,t}=\left\{ \begin{array}{cc} 0.4 & t\in[1,T^{0.6}] \\
    0.4+ 0.1\times (-1)^t t^{-1/4}  & t\in(T^{0.6},T] \end{array}.\right.
\end{align*}
We select $T=1000$ and thus $T^{0.6}\approx63$.
It can be verified that the selection of parameters yields $m_t\geq L_t \sqrt{N-1}$ for $\forall t$, and $V_T = \mathcal{O}(T^{3/4})$, $W_{i,T} = \mathcal{O}(T^{3/4})$, $i=1,2$. Figures \ref{fig_tracking}--\ref{fig_DR}  illustrate the equilibrium tracking  error and the dynamic regret of the best response algorithm, respectively. We observe that, when $t\in[1,T^{0.6}]$, both the equilibrium tracking error and the dynamic regret grow rapidly due to the oscillations of $b_{i,t}$; when $t\in(T^{0.6},T]$, they grow slowly since $b_{i,t}$ is a constant and the variation of $e_{i,t}$ is decreasing over time.
Moreover, both the equilibrium tracking error and the dynamic regret are sub-linear in the total number of episodes, which supports our theoretical results.

\begin{figure}[t]
\begin{center}
\centerline{\includegraphics[width=0.9\columnwidth]{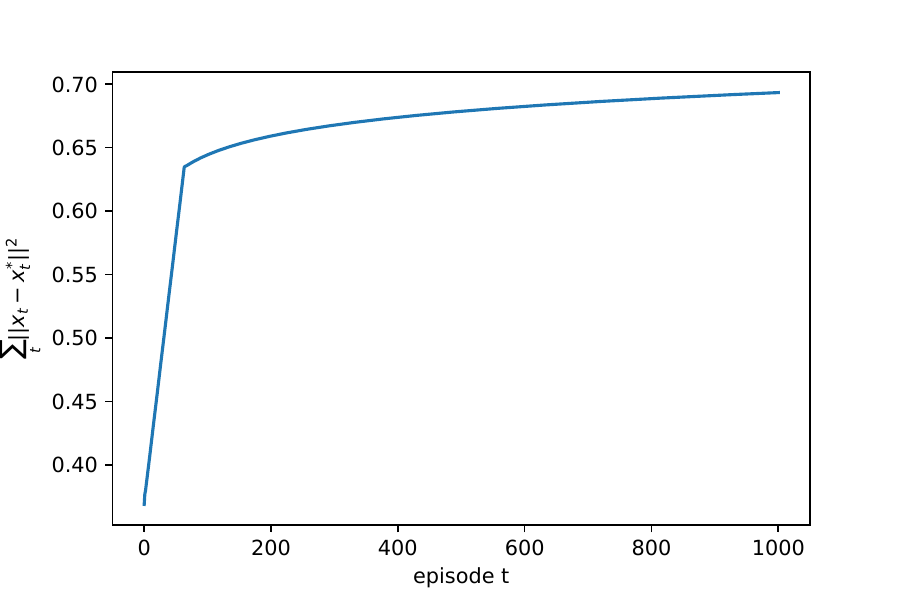}}
\caption{Equilibrium tracking error of the best response algorithm for time-varying games. }
\label{fig_tracking}
\end{center}
\vskip -0.3in
\end{figure}

\begin{figure}[t]
\begin{center}
\centerline{\includegraphics[width=0.9\columnwidth]{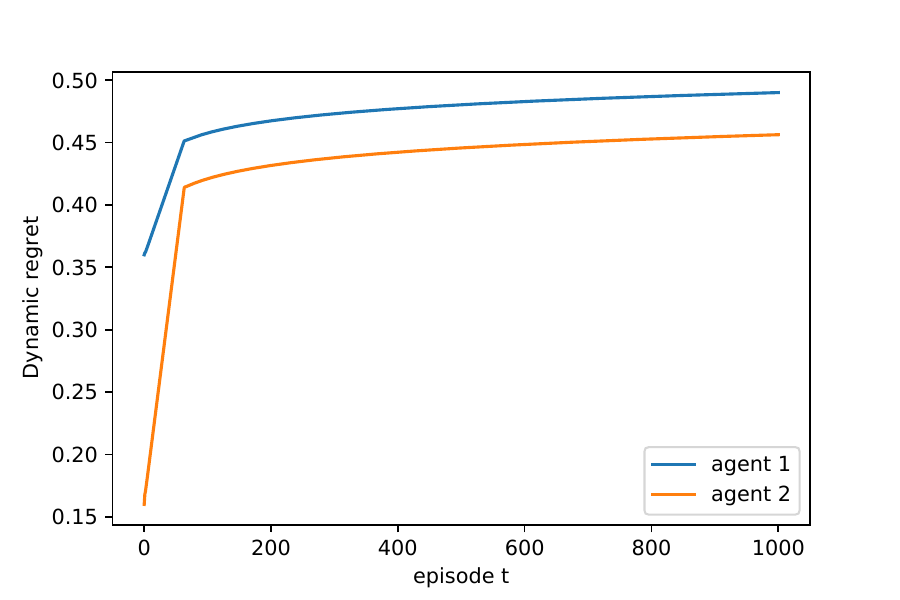}}
\caption{Dynamic regret of the best response algorithm for time-varying games.}
\label{fig_DR}
\end{center}
\vskip -0.3in
\end{figure}

\section{Conclusion}\label{sec:conclusion}
In this work, we analyzed the best response algorithm for the class of strongly monotone games. We first considered standard time-invariant games and obtained a sufficient condition under which the best response algorithm converges at an exponential rate.
We provided numerical experiments that showed the best response algorithm can diverge if this condition fails to hold, which indicates that the condition is tight. Subsequently, we analyzed the best response algorithm for time-varying games with evolving equilibria. We showed that the equilibrium tracking error and the dynamic regret can be bounded in terms of the variations of evolving equilibria and loss functions. 
Moreover, we provided additional numerical simulations to verify our results.

\bibliography{ref}
\bibliographystyle{IEEEtran}

\end{document}